\documentclass[reqno,11pt]{amsart}
\usepackage{amsthm,amsfonts,amssymb,euscript,bbm,amsmath,mathrsfs}

\usepackage{empheq}
\usepackage{latexsym}
\usepackage{enumerate}
\usepackage{array}
\usepackage{cancel}
 \usepackage{color,ulem,textcomp}
\usepackage{hyperref}
\usepackage{pdfsync}

\newcommand{\dis}{\displaystyle}
\newcommand{\bea}{\begin{eqnarray}}
\newcommand{\eea}{\end{eqnarray}}
\def\beaa{\begin{eqnarray*}}
\def\eeaa{\end{eqnarray*}}
\def\ba{\begin{array}}
\def\ea{\end{array}}
\def\be#1{\begin{equation} \label{#1}}
\def \eeq{\end{equation}}

\def\eps{\varepsilon}
\def\DD{{\mathcal D}}
\def\R{{\mathbb{R}}}
\def\C{{\mathbb{C}}}

\def\Z{{\mathbb{Z}}}
\def\T{{\mathbb{T}}}
 \newcommand{\W}{\mathcal W}
\newcommand{\Vc}{\mathcal V}

\renewcommand{\>}{  \rangle   }

\newtheorem{theorem}{Theorem}[section]
\newtheorem{lemma}[theorem]{Lemma}
\newtheorem{proposition}[theorem]{Proposition}
\newtheorem{corollary}[theorem]{Corollary}

\newtheorem{assumption}[theorem]{Assumption}

\numberwithin{equation}{section}

\definecolor{gr}{rgb}   {0.,   0.69,   0.23 }
\definecolor{bl}{rgb}   {0.,   0.5,   1. }
\definecolor{mg}{rgb}   {0.85,  0.,    0.85}
\definecolor{yl}{rgb}   {0.8,  0.7,   0.}
\definecolor{or}{rgb}  {0.7,0.2,0.2}

\catcode`\Ž = \active\def Ž{\'e} \catcode`\ƒ = \active\def ƒ{\'E}  \catcode`\Ë = \active\def Ë{\`A}\catcode`\ = \active\def {\`e}
\catcode`\ = \active\def {\c{c}} \catcode`\ˆ = \active\def
ˆ{\`a} \catcode`\ = \active \def {\^e} \catcode`\ = \active
\def {\`u} \catcode`\™ = \active \def ™{\^o} \catcode`\ž =
\active \def ž{\^u} \catcode`\" = \active \def "{\^{\i}}
\catcode`\‰ = \active \def ‰{\^a} \catcode`\š = \active \def
š{{\"o}}

\begin{document}

\author{Beno\^it Gr\'ebert, \'Eric Paturel and Laurent Thomann}
\address{Laboratoire de Math\'ematiques J. Leray, UMR  6629 du CNRS, Universit\'e de Nantes, 
2, rue de la Houssini\`ere,
44322 Nantes Cedex 03, France}
\email{benoit.grebert@univ-nantes.fr}
\email{eric.paturel@univ-nantes.fr}
\email{laurent.thomann@univ-nantes.fr}

\thanks{B.G. and L.T. are partially supported   by the  grant  ``ANA\'E'' ANR-13-BS01-0010-03}

\title[Modified Scattering on $\mathbb{R}\times\mathbb{T}^d: $ the nonresonant case]{Modified scattering for the cubic Schr\"odinger equation on product spaces: the nonresonant case}


\begin{abstract}
We consider the cubic nonlinear Schr\"odinger equation   on the spatial domain $\R\times \T^d$, and we perturb it with a convolution potential.  Using recent techniques of  Hani-Pausader-Tzvetkov-Visciglia, we   prove a modified scattering result and construct modified wave operators, under generic assumptions on the potential. In particular, this enables us to prove that the Sobolev norms of small solutions of this nonresonant cubic NLS are asymptotically constant.   
\end{abstract}
\maketitle


\section{Introduction}

\subsection{Motivation and backgrounds}
In the last  years, much effort has been done   to understand the weak turbulence phenomenon  in  Hamiltonian \ nonlinear dispersive PDEs. The central question is the following: once we have proved the global well posedness of a PDE in a Sobolev space~$H^{s_0}$, we want to know whether
\begin{itemize}
\item[$(i)$]  the  solutions remain bounded for all time and in all Sobolev norms, i.e. $$\|u(t)\|_s\leq C_s \|u(0)\|_s,\ \forall s\geq s_0$$ at least for small initial conditions (a strong stability results for the origin), 
\item[$(ii)$] there exist initial conditions leading to unbounded solutions, i.e. $$\exists\, u(0) \text{ such that } \limsup_{t\to +\infty} \|u(t)\|_s=+\infty \text{ for some } s\geq s_0.$$
\end{itemize}

 The first significant result in direction~$(ii)$ is due to  Bourgain~\cite[Section~4]{Bo1996c} who showed a polynomial growth of Sobolev norms for a nonlinear wave equation in 1d with periodic boundary conditions. Later on,  Colliander-Keel-Staffilani-Takaoka-Tao (see~\cite{CKSTT}),  considered  the cubic  nonlinear Schr\"odinger equation,   on the two dimensional torus $\T^{2}=\big(\R/(2\pi\Z)\big)^{2}$
 \begin{equation} \label{NLST2} i\partial_t u +\Delta u=\vert u\vert^2u , \quad (t,x)\in \R\times  \T^{2}\end{equation}
and proved that  for any $K \geq 1$ there exists a solution $u$ and a time $T $ such that $\|u(T )\|_s \geq K\|u(0 )\|_s$. Of course, this result is weaker than the assertion~$(ii)$ but it suggests a possible unbounded behavior for some solutions. \  After that, Guardia-Kaloshin (see~\cite{GK}), improving the dynamical step, proved that the time $T$ satisfies a polynomial bound $0 < T < K^c$ for some absolute constant $c>0$. A maybe less intuitive extension is then obtained by  M. Guardia (see~\cite{G}): he proves that this "almost unbounded" behavior is not a consequence of the exact resonances in~\eqref{NLST2}, since it persists when one adds a small convolution potential $V$:
\begin{equation} \label{NLST2V} i\partial_t u +\Delta u+ V\star u=\vert u\vert^2u , \quad (t,x)\in \R\times  \T^{2}.\end{equation}
 In fact, in~\cite{CKSTT} (resp. in~\cite{GK, G}) the authors proved that the solutions of~\eqref{NLST2} (resp.~\eqref{NLST2V}) remain close to the solution of a finite dimensional (depending on $K$) resonant system and they constructed an explicit solution $v_K$ of this finite dimensional dynamical system (which also depends on $K$) satisfying $\|v_K(T )\|_s \geq K\|v_K(0 )\|_s$.\\
 However we could expect that, since the potential $V$ generically kills the exact resonances, the solutions of\;\eqref{NLST2V} would not follow the resonant dynamics. Actually in a series of paper initiated by~\cite{Bam03, BG}, Bambusi-Gr\'ebert  developed a Birkhoff normal form technic that shows that, in the context of~\eqref{NLST2V}, assertion $(i)$ is {\it almost} satisfied for a generic choice of $V$. Precisely they proved a stability result of kind $(i)$  for $t\leq C \eps^{-M}$ where $\eps=\|u_0\|_s\ll 1$ and $M$ is an arbitrary constant fixed from the beginning (see also \cite{BG2, BDGS, GIP} for developments  or~\cite{Bam07, Gre} for a simple presentation). \\
Notice that this stability result is even stronger in analytic regularity as conjectured in~\cite{Bo2004b} and proved in \cite{FG}:  if the initial datum is analytic in a strip then the solution is bounded in a strip of half width during a time of order $ \eps^{-\sigma|\ln \eps|^\beta}$ where $\eps>0$ is the initial size of the solution and $0<\beta<1$.
Surprisingly, the result in~\cite{G} shows that the resonant behavior in~\eqref{NLST2V} may coexist with these almost stability results.\

\medskip
 Let us also mention some interesting phenomena concerning the periodic Szeg\"o equation introduced by G\'erard and Grellier\;\cite{GG2}.  
Recently, in\;\cite{GGX}  they showed the alternative\;$(ii)$ for generic initial conditions, despite of an infinite number of conservation laws. Concerning the Szeg\"o equation on the real line, Pocovnicu\;\cite{Pocovnicu}  proved\;$(ii)$ by giving an explicit example.\
\medskip

More recently Hani-Pausader-Tzvetkov-Visciglia considered in~\cite{HPTV} the cubic nonlinear Schr\"odinger equation on the wave-guide manifolds $\R\times \T^d$
 \begin{equation} \label{NLST2R}
 i\partial_t u +\Delta_{\mathbb{R}\times\mathbb{T}^d} u=\vert u\vert^2u , \quad (t,x,y)\in \R\times \R\times \T^{d},
\end{equation}
so they added a direction of diffusion in the PDE. Due to the dispersion along one variable, we expect that this equation is less "turbulent" than~\eqref{NLST2}.  Actually   they proved that in the case $d=1$ the equation~\eqref{NLST2R}   satisfies the assertion~$(i)$ in the alternative above, and when  $2\leq d\leq 4$ it satisfies the assertion~$(ii)$.  \ \medskip

In this work we add a convolution potential $V$ to~\eqref{NLST2R}, i.e. we consider
 \begin{equation}\label{NLSP}  
 i\partial_t u +\Delta_{\mathbb{R}\times\mathbb{T}^d} u+V\star  u=\vert u\vert^2u , \quad (t,x,y)\in \R\times \R\times \T^{d}
 \end{equation}
 and we prove that for generic choice of the potential $V$ assertion $(i)$ holds true. So in that "less turbulent" case, the exact resonances are determinant to decide the limit dynamics: when we kill the exact resonances we turn off the weak turbulence phenomenon.  As proven in\;\cite{HPTV}, in the case $d=1$, the resonances are trivial, and this leads to $(i)$. One difficulty in the study of\;\eqref{NLST2R} and of\;\eqref{NLSP} is that the nonlinearity is long range and thus may induce strong nonlinear interactions.  \ Here the range has to be computed with respect to the dimension of the Euclidian component of the domain: hence a cubic nonlinearity is long range on $\R^{d'} \times \T^{d}$ for  $d' \leq 1$. Let us recall  the heuristics which leads to define the notion of short and long range of the nonlinearity $|u|^{p}u$. If one believes that the solution of NLS decays like the linear evolution  group ($\|e^{it \Delta_{\R^{d'}\times \T^{d}}}u_{0}\|_{L^{\infty}} \sim  t^{-d'/2}$ when $t\longrightarrow+\infty$), then one says that the nonlinearity is short range  if the potential $|u|^{p}\sim t^{-p d'/2}$ is integrable at infinity. 
 
The control of higher order Sobolev norms (i.e. assertion $(i)$) in the case of short range nonlinearities may be obtained by global in time Strichartz inequalities (see e.g. \cite[p.7, Theorem~2]{BourgainBook}).  In\;\cite{TzVi2}, Tzvetkov and Visciglia recently proved scattering results, with large initial conditions, on $\R^d \times \T$. This shows that $(i)$ may also hold true on product manifolds and for large initial conditions. \ For long range nonlinearities on Euclidean spaces, given initial data of arbitrary size (at least in the defocusing case), it is possible to obtain polynomial bounds  (\cite{Bo1996c, Bo1998b, St1997, St1997b, CoDeKnSt, So2011})  like ${\mathcal{O}}(t^{\alpha(s-1)})$ for the $H^{s}$ norm, with $s >1$ and $\alpha >0$ depending on the context, a notable exception being integrable NLS  (cubic NLS on $\R$), where these norms are bounded in time. On compact domains, such studies for NLS give rise to similar polynomial bounds (\cite{Bo1996c, St1997b, So2011b, CoKwOh2012}). Our main result in this paper is to prove  assertion $(i)$ under a smallness assumption on the initial data. We guess that an adaptation of the {\it upside-down} I-method, which gave some of the most accurate results quoted before, could be done in our context and give polynomial bounds for  Sobolev norms of any order, without smallness assumption on the initial data.

Finally, observe that even for linear Schr\"odinger equations on compact manifolds we only have in general    subpolynomial bounds (${\mathcal{O}}(t^{\eps})$ for every $\eps>0$, or under analytic assumptions logarithmic bounds). See  \cite{Bo1999, Bo2004, Zh2008, De2011}. 
\medskip

In~\cite{HPTV}, the  proof consists in establishing a modified scattering and in constructing modified wave operators. It turns out that the modified asymptotic dynamics are dictated by the resonant part of~\eqref{NLST2R} and that this resonant system has solutions with infinitely growing high Sobolev norms $H^s$. In our case, we can follow  the same strategy but, since we add the convolution potential $V$, the modified asymptotic dynamics are dictated by a  non resonant system which does not allow interaction between different energy levels.

\medskip

Notice that when one adds a second direction of diffusion, i.e. considering~\eqref{NLST2R} on $\R^2\times \T^{d}$, then the solutions scatter to constant solutions (see Tzvetkov-Visciglia~\cite{TzVi}) and thus we are again in case\;$(i)$, which is coherent with the short range of the nonlinearity. So~\eqref{NLST2R} on $\R\times \T^{d}$ seems to be a limit case with respect to the alternative above. In this perspective, we can conjecture that~\eqref{NLST2}  is weak turbulent in the sense of~$(ii)$ (actually more turbulent than~\eqref{NLST2R}). The case of~\eqref{NLST2V} is less clear, in particular in view of the existence of plenty of linearly stable KAM tori proved in~\cite{EK}.

\subsection{Statement of the result}
Denote by $\T^{d}=\big(\R/(2\pi\Z)\big)^{d}$. In this   this work we  study the asymptotic behavior of the cubic defocusing nonlinear Schr\"odinger equation posed on the wave-guide manifolds $\R\times \T^d$, 
 \begin{equation} \label{NLS}
  \left\{
      \begin{aligned}
      &i\partial_t U +\Delta_{\mathbb{R}\times\mathbb{T}^d} U+V\star  U=\vert U\vert^2U , \quad (t,x,y)\in \R\times \R\times \T^{d},
       \\  &  U(0,x,y)  =U_0(x,y),
      \end{aligned}
    \right.
\end{equation}
where the unknown $U$ is a complex-valued function, and where $V$ is a generic perturbation which only depends on the variable $y$. In the sequel we denote by $\DD$ the whole linear operator
 \begin{equation*}
 \DD=\Delta_{\R\times\mathbb{T}^d} +V\star=\partial^2_x+\Delta_{\mathbb{T}^d} +V\star.
 \end{equation*}
 For $p\in \Z^{d}$, we denote by $\hat{V}_{p}$ the Fourier coefficients of $V$. The eigenvalues of the operator $\Delta_{\mathbb{T}^d} +V\star$ are 
\begin{equation*}
\lambda_{p}:=-|p|^{2}+\hat{V}_{p}, \quad p\in\Z^d.
\end{equation*}
In this paper we assume that $V$ belongs to the following space ($m> d/2$, $R>0$)
\begin{equation}\label{pot}
\W_{m}=\big\{ V(x)=\sum_{a\in \Z^d}v_a
e^{i a\cdot x}
\mid v'_{a}:=\frac{{v_a}{(1+|a|)^m}}{R} \in [-1/2,1/2]
\mbox{ for any }a\in \Z^d\; \big\}
\end{equation}
that we endow with the product probability measure\footnote{Here, for $a = (a_1,\ldots,a_d) \in \Z^d$, $|a|^2 = a_1^2 + \cdots + a_d^2$.}. 
 
 \medskip

In the sequel we suppose that the following non resonance assumption is satisfied
  \begin{assumption}\label{As1}
{\bf (Non resonance assumption)}:  There exist $c>0$ and $\gamma>0$ such that  for all $(p,q,r,s)\in \Z^{d}$ with $\big\{\vert p\vert, \vert r\vert\big\}\neq\big\{\vert q\vert,\vert s\vert\big\}$ one has
\begin{equation}\label{res1}
\big| \lambda_p-\lambda_q+\lambda_r-\lambda_s \big|\geq \frac{c}{\nu_{3}(p,q,r,s)^{\gamma}}
\end{equation}
where $\nu_{3}(p,q,r,s)$ is the third largest number among $|p|,|q|,|r|,|s|$.
\end{assumption}
This condition means that if $\big| \lambda_p-\lambda_q+\lambda_r-\lambda_s \big|$ is small, then at least three terms among $\big\{ |p|, |q|, |r|, |s|\big\}$ are large. Such a condition is well-adapted to control quadri-linear terms (see the proof of Lemma~\ref{FL}).

 It turns out that Assumption~\ref{As1} is generic in the following sense:
\begin{lemma}\label{generic}
Fix $m>d/2$, $R>0$. There exists a set $\Vc_m \subset \W_m$ of measure 1 such that, for any 
$\mu \in \mathcal{V}_m$,   Assumption \ref{As1} holds true.
\end{lemma}
The proof is quite standard and is a consequence of~\cite[Proposition 2.7]{FG}  (see also~\cite{BG}).  

\medskip

 We now  define the limit system
\begin{equation}\label{RSS}
i\partial_\tau G(\tau)=\quad\mathcal R[G(\tau), G(\tau), G(\tau)],
\end{equation}
where
\begin{equation*}
\mathcal F_{\R\times \T^d}\, \mathcal R[G, G, G](\xi, p)=\quad\sum_{\substack{p_1+p_3=p+p_2\\ \{\vert p_1\vert, \vert p_{3}\vert\}=\{\vert p\vert,\vert p_2\vert\}}}\widehat{G}(\xi,p_1)\overline{\widehat{G}(\xi,p_2)}\widehat{G}(\xi,p_3).
\end{equation*}
Here $\widehat G(\xi, p)=\mathcal F_{\R\times \T^d} G (\xi, p)$ is the Fourier transform of $G$ at $(\xi, p)\in \R\times \Z^d$. Observe  that the dependence on $\xi$ is merely parametric. The system~\eqref{RSS} is   the resonant system for the cubic NLS equation on $\mathbb{T}^d$, with the operator $\Delta_{\mathbb{T}^d} +V\star$, provided that the non resonant assumption~\eqref{res1} is satisfied.

In the sequel we fix $N_{0}=N_{0}(d,\gamma)$ a large integer which will be given by the proof, which only depends on the dimension~$d$ and on the parameter $\gamma>0$ which appears in~\eqref{res1}. For  $N\geq N_{0}$ we define the Banach spaces $S$ and $S^{+}$ by the norms

\begin{equation*} 
 \Vert F\Vert_{S}:=\Vert F\Vert_{H^N_{x,y}}+\Vert xF\Vert_{L^2_{x,y}},\quad
\Vert F\Vert_{S^+}:=\Vert F\Vert_S+\Vert (1-\partial_{xx})^4F\Vert_{S}+\Vert xF\Vert_{S}.
 \end{equation*}

Following the same line as in~\cite{HPTV}, we prove that the solutions of~\eqref{NLS} scatter to solutions of the resonant system~\eqref{RSS}:
\begin{theorem}\label{thm1}
Let $ 1\leq d \leq 4$ and $N\geq N_0$. There exists $\eps=\eps(N,d)>0$ such that if $U_0\in S^+$ satisfies
\begin{equation*}
\Vert U_0\Vert_{S^+}\le\eps,
\end{equation*}
and if $U(t)$ solves~\eqref{NLS} with initial data $U_0$, then $U\in \mathcal{C}((0,+\infty);H^N)$ exists globally and exhibits modified scattering to its resonant dynamics~\eqref{RSS} in the following sense: there exists $G_0\in S$ such that if $G(t)$ is the solution of~\eqref{RSS} with initial data $G(0)=G_0$, then
\begin{equation*}
\Vert U(t)-e^{it\DD}G(\pi\ln t)\Vert_{H^N(\mathbb{R}\times\mathbb{T}^d)}\longrightarrow 0\quad\hbox{ as }\quad t\longrightarrow+\infty,
\end{equation*}
and
\begin{equation*}
\|U(t)\|_{L^\infty_x H^1_y}\leq C (1+|t|)^{-\frac{1}{2}}.
\end{equation*}
\end{theorem}

At this stage we observe that  the dynamics of~\eqref{RSS} are bounded (see Lemma~\ref{gAdmi}). Actually~\eqref{RSS} is globally well-posed in $H^{1}_{x,y}$ for $1\leq d \leq 4$,  and all $H^{N}_{x,y}$ norms are conserved by the flow:
 \begin{equation*}
 \|G(t)\|_{H^{N}_{x,y}} =   \|G_{0}\|_{H^{N}_{x,y}}.
 \end{equation*}
As a consequence we obtain our main result 
\begin{corollary}\label{coro14}
Let $ 1\leq d \leq 4$ and $N\geq N_0$. There exists $\eps=\eps(N,d)>0$ such that if $U_0\in S^+$ satisfies
\begin{equation*}
\Vert U_0\Vert_{S^+}\le\eps,
\end{equation*}
and if $U(t)$ solves~\eqref{NLS} with initial data $U_0$, then $U\in \mathcal{C}((0,+\infty);H^N)$ exists globally and
$$
\|U(t)\|_{ H^N_{x,y}}\leq   C_N\,\eps\,
$$
for some constant $C_N$ depending only on $V$ and $N$. Moreover, $\|U(t)\|_{ H^N_{x,y}}$ tends to a constant when $t\longrightarrow +\infty$.
\end{corollary}

This shows that every solution to \eqref{NLS} issued from a small and smooth initial condition has  asymptotically constant Sobolev norms.\medskip

We also notice that, as in~\cite{HPTV}, we can construct modified wave operators in the following sense:

\begin{theorem}\label{thm2}
Let $ 1\leq d \leq 4$ and $N\geq N_0$.  There exists $\eps=\eps(N,d)>0$ such that if $G_0\in S^+$ satisfies
\begin{equation*}
\Vert G_0\Vert_{S^+}\le\eps,
\end{equation*}
and $G(t)$ solves~\eqref{RSS} with initial data $G_0$, then there exists $U\in \mathcal{C}((0,\infty);H^N)$ a solution of~\eqref{NLS} such that
\begin{equation*}
\Vert U(t)-e^{it\DD}G(\pi\ln t)\Vert_{H^N(\mathbb{R}\times\mathbb{T}^d)}\longrightarrow0\quad \hbox{ as }\quad t \longrightarrow+\infty.
\end{equation*}
\end{theorem}
 There are analogue statements in the limit $t\longrightarrow -\infty$.\medskip
 
 As we mentioned previously, the analogues of Theorems~\ref{thm1} and~\ref{thm2} in the case $V=0$ were proved in~\cite{HPTV} (see also~\cite{HPTVproc}). We  show here that the same strategy as~\cite{HPTV} also applies in a case where there are small divisors. Most of the arguments of~\cite{HPTV} apply mutatis mutandis and we will rely on them.  In this text, we focus on  the differences, namely on the control of the terms containing small divisors. 
 
 In~\cite{HPTV}, the regularity condition was $N_{0}=30$. Here, the corresponding $N_{0}$  is not explicit, and possibly large: it depends on $\gamma$ which appears in~\eqref{res1} and Lemma~\ref{generic}.  It would be interesting  to understand what happens to less regular initial conditions, namely the case $N_0\geq 1$. But this seems to be a very difficult question.\ \medskip
 
 It is likely that in the previous statements we can avoid the restriction $d\leq 4$. This assumption was needed in~\cite{HPTV}, because this was the condition such that the corresponding limit system was well-posed in the energy space, namely $H^{1}$. Here instead, we can use that every $H^{s}$-norm is invariant by the flow of the limit system~\eqref{RSS}, and we expect that we can follow the analysis in~\cite{HPTV} and replace $H^{1}$ by $H^{s}$ for $s>d/2$. 

\subsection{Notations} For the reader's convenience, we keep most of the notations used in~\cite{HPTV}, and we recall them below.\medskip

$\bullet$ {\it Fourier transforms and frequency localisation:} We define the Fourier transform on $\mathbb{R}$ by
\begin{equation*}
\widehat{g}(\xi):=\frac{1}{2\pi}\int_{\mathbb{R}}e^{-ix\xi}g(x)dx.
\end{equation*}
Similarly, if $F(x,y)$ depends on $(x,y)\in\mathbb{R}\times\mathbb{T}^d$, $\widehat{F}(\xi,y)$ denotes the partial Fourier transform in\;$x$. The Fourier coefficient of $h:\mathbb{T}^d\to\mathbb{C}$ is given by 
\begin{equation*}
h_p:=\frac{1}{(2\pi)^{d}}\int_{\mathbb{T}^d}h(y)e^{-i\langle p,y\rangle}dy,\qquad p\in\mathbb{Z}^d.
\end{equation*}
The full  spacial Fourier transform reads
\begin{equation*}
\left(\mathcal{F}F\right)(\xi,p)=\frac{1}{(2\pi)^{d}}\int_{\mathbb{T}^d}\widehat{F}(\xi,y)e^{-i\langle p,y\rangle}dy=\widehat{F}_p(\xi).
\end{equation*}

 We define the Littlewood-Paley projections in the $x$ variable by 
\begin{equation*}
\left(\mathcal{F}Q_{\leq N}F\right)(\xi,p)=\varphi(\frac{\xi}{N})\left(\mathcal{F}F\right)(\xi,p),
\end{equation*}
where $\varphi\in \mathcal{C}^\infty_c(\mathbb{R})$, $\varphi(x)=1$ when $\vert x\vert\leq 1$ and $\varphi(x)=0$ when $\vert x\vert\ge 2$. Next, we define
\begin{equation*}
Q_N=Q_{\leq N}-Q_{\leq N/2},\quad Q_{\ge N}=1-Q_{\leq N/2}.
\end{equation*}
$\bullet$ {\it Resonant sets:}
We define the zero momentum set by  
\begin{equation}\label{DiscPar}
\mathcal{M}:=\big\{(p,q,r,s)\in\mathbb{Z}^{4d}:\,\,p-q+r-s=0\big\},\\
\end{equation}
and the resonant  set by
\begin{equation*}
\Gamma_0:=\big\{(p,q,r,s)\in\mathcal{M}:\,\,\lambda_{p}-\lambda_{q}+\lambda_{r}-\lambda_{s}=0\big\}.
\end{equation*}
Under Assumption \ref{As1} on $(\mu_ j)_{j\geq 0}$ we have 
\begin{equation*}
\Gamma_0=\big\{(p,q,r,s)\in\mathcal{M}:\,\, (\vert p\vert=\vert q\vert \;\; \text{and}\;\; \vert r\vert=\vert s\vert)  \;\; \text{or}\;\;  (\vert p\vert=\vert s\vert \;\; \text{and}\;\; \vert q\vert=\vert r\vert) \big\}.
\end{equation*}

$\bullet$ {\it Structure of the nonlinearity:} Let us define the trilinear form $\mathcal{N}^{t}$  by 
\begin{equation}\label{defN}
\mathcal{N}^{t}[F,G,H]:=e^{-it\DD}\Big( e^{it\DD}F\cdot e^{-it\DD}\overline{G}\cdot e^{it\DD}H\Big).
\end{equation}
Let $\dis U(t,x,y)=e^{it\DD} F(t)$, then we see that $U$ solves~\eqref{NLS} if and only if $F$ solves
\begin{equation*} 
i\partial_t F(t) = \mathcal{N}^{t}[F(t),F(t),F(t)].
\end{equation*}
A direct computation shows that
\begin{multline*}
\mathcal{F}\mathcal{N}^t[F,G,H](\xi,p)
=
\sum_{(p,q,r,s)\in\mathcal{M}}e^{it\left[\lambda_{p}-\lambda_{q}+\lambda_{r}-\lambda_{s}\right]}
\int_{\mathbb{R}^2}e^{it2\eta\kappa}\widehat{F}_{q}(\xi-\eta)\overline{\widehat{G}_{r}}(\xi-\eta-\kappa)\widehat{H}_{s}(\xi-\kappa)d\kappa d\eta\,.
\end{multline*}
The resonant part of the nonlinearity is defined by   
\begin{equation}\label{DeF}
\mathcal{F}\mathcal{R}[F,G,H](\xi,p):=\sum_{(p,q,r,s)\in\Gamma_0}\widehat{F}_q(\xi)\overline{\widehat{G}_r}(\xi)\widehat{H}_s(\xi).
\end{equation}
 
$\bullet$ {\it Norms:} We consider the following Sobolev norm on sequences  
\begin{equation*}
\Vert \{a_p\}\Vert_{h^s_p}^2:=\sum_{p\in\mathbb{Z}^d}\left[1+\vert p\vert^2\right]^s\vert a_p\vert^2.
\end{equation*}
In the sequel we will need the norm on functions $F: \R\times \T^d\longrightarrow \C$
\begin{equation*}
\Vert F\Vert_{Z}^2:=\sup_{\xi\in\mathbb{R}}\left[1+\vert \xi\vert^2\right]^2\sum_{p\in\mathbb{Z}^d}\left[1+\vert p\vert^{2}\right]\vert\widehat{F}_p(\xi)\vert^2=\sup_{\xi\in\mathbb{R}}\left[1+\vert \xi\vert^2\right]^2\Vert \widehat{F}_p(\xi)\Vert_{h^1_p}^2.
\end{equation*}

 \section{Structure of the nonlinearity} 
 In this section we explain how we can adapt the method of~\cite{HPTV} in order to prove Theorems~\ref{thm1} and~\ref{thm2}. The first step is to understand well the structure of the nonlinearity $\mathcal{N}^{t}$ in~\eqref{defN}, and this is the content of Proposition~\ref{FL1} below. With this result at hand,  Theorems~\ref{thm1} and~\ref{thm2} are proven exactly as in~\cite[Sections~5 and 6]{HPTV}   by fixed point arguments. Therefore we only focus on the proof of Proposition~\ref{FL1}. \medskip

We show here that the nonlinear term can be decomposed into an effective term, plus a remainder, which is - roughly speaking- integrable in time. Namely we can write
\begin{equation*} 
\mathcal{N}^t[F,G,H]=\frac{\pi}{t}\mathcal{R}[F, G,H]+\mathcal{E}^t[F,G,H]
\end{equation*}
where $\mathcal{R}$ is given in~\eqref{DeF}. 

For some small enough  absolute constant $\delta>0$, we define the space-time norms  
\begin{eqnarray*} 
\Vert F\Vert_{X_T}&:=&\sup_{0\leq t\leq T}\big\{\Vert F(t)\Vert_Z+(1+\vert t\vert)^{-\delta}\Vert F(t)\Vert_S +(1+\vert t\vert)^{1-3\delta}\Vert\partial_t F(t)\Vert_S\big\},\\
\Vert F\Vert_{X_T^+}&:=&\Vert F\Vert_{X_T}+\sup_{0\leq t\leq T}\big\{(1+\vert t\vert)^{-5\delta}\Vert F(t)\Vert_{S^+}+(1+\vert t\vert)^{1-7\delta}\Vert\partial_t F(t)\Vert_{S^+}\big\}.
\end{eqnarray*}

The next result is an analogue of~\cite[Proposition 3.1]{HPTV} and reads as follows

\begin{proposition}\label{FL1}
Assume that for $T\geq 1$,  $F,G,H: \mathbb{R}\to S$ satisfy
\begin{equation} \label{le}
\Vert F\Vert_{X_{T}}+\Vert G\Vert_{X_{T}}+\Vert H\Vert_{X_{T}}\leq 1.
\end{equation}
Then for $t\in [T/4,T]$, we can write
$$
\mathcal{E}^t[F(t),G(t),H(t)]=\mathcal{E}_1^t+\mathcal{E}_2^t\,,
$$
where the following bounds hold uniformly in $T\geq 1$,
\begin{eqnarray*} 
T^{-\delta}\Vert \int_{T/2}^T\mathcal{E}_i(t)dt\Vert_S&\lesssim &1,\quad i=1,2,\\
T^{1+\delta}\sup_{T/4\leq t\leq T}\Vert \mathcal{E}_1(t)\Vert_Z&\lesssim &1,\\
T^{\frac{1}{10}}\sup_{T/4\leq t\leq T}\Vert \mathcal{E}_3(t)\Vert_S&\lesssim& 1,
\end{eqnarray*}
where $\mathcal{E}_2(t)=\partial_t\mathcal{E}_3(t)$.
Assuming in addition
\begin{equation} \label{BA+}
\Vert F\Vert_{X^+_{T}}+\Vert G\Vert_{X^+_{T}}+\Vert H\Vert_{X^+_{T}}\leq 1,
\end{equation}
we also have that
\begin{equation*} 
T^{-5\delta}\Vert \int_{T/2}^T\mathcal{E}_i(t)dt\Vert_{S^+}\lesssim 1,\qquad  T^{2\delta}\Vert \int_{T/2}^T\mathcal{E}_i(t)dt\Vert_S\lesssim 1,\quad i=1,2.\\
\end{equation*}
\end{proposition}
We now explain how we can prove this result. 

To begin with, for $T\geq 1$, we decompose the nonlinearity $\mathcal{N}^{t}$ according to the high and the low frequencies in the~$x$-variable
\begin{equation*}
\mathcal{N}^{t}=\sum_{\substack{A,B,C\\\max(A,B,C)\ge T^{\frac{1}{6}}}}\mathcal{N}^t[Q_AF,Q_BG,Q_CH]+\sum_{\substack{A,B,C\\\max(A,B,C)\leq T^{\frac{1}{6}}}}\mathcal{N}^t[Q_AF,Q_BG,Q_CH].
\end{equation*}
The first term is treated in~\cite[Lemma 3.2]{HPTV}. We turn to the second one, and in the sequel we assume that 
\begin{equation}\label{tronca}
F=Q_{\leq T^{1/6}}F, \quad G=Q_{\leq T^{1/6}}G, \quad H=Q_{\leq T^{1/6}}H\,.
\end{equation}
 We decompose the second term by taking into account   the resonances w.r.t to the~$y$-variable, namely
\begin{equation*} 
\sum_{\substack{A,B,C\\\max(A,B,C)\leq T^{\frac{1}{6}}}}\mathcal{N}^t[Q_AF,Q_BG,Q_CH]= \mathcal{N}_{0}^t[F,G,H]+ \mathcal{N}_{nr}^t[F,G,H],
 \end{equation*}
where $\mathcal{N}_{0}^t$ is defined by 
\begin{equation*} 
 \mathcal{F} \mathcal{N}_{0}^t[F,G,H](\xi,p):=\sum_{(p,q,r,s)\in\Gamma_0}\int_{\mathbb{R}^2}e^{it2\eta\kappa}\widehat{F}_{q}(\xi-\eta)\overline{\widehat{G}_{r}}(\xi-\eta-\kappa)\widehat{H}_{s}(\xi-\kappa)d\kappa d\eta\,.
 \end{equation*}
 The quantity $\mathcal{N}_{0}^t$ contains the resonant interactions of the nonlinearity. Observe that under Assumption~\ref{As1} there are much fewer resonances than in the case $V=0$, therefore the analysis of~\cite{HPTV} also applies in our context. More precisely, the arguments of~\cite[Lemma 3.7 and Remark~3.8]{HPTV}   show that this term can be written 
\begin{equation*}
\mathcal{N}_{0}^t[F,G,H]=\frac{\pi}{t}\mathcal{R}[F, G,H]+\mathcal{E}^t[F,G,H],
\end{equation*}
where $\mathcal{E}^t$ satisfies the estimates of Proposition~\ref{FL1}. We also refer to the end of the proof of~\cite[Proposition~3.1]{HPTV} for more details.\medskip

The contribution of $\mathcal{N}_{nr}^t$ is slightly different in our case than in the case considered in~\cite{HPTV}, because of the presence of small denominators. In this context, we are able to  prove the following

\begin{lemma}\label{FL}
For $T\geq 1$,  assume that $F$, $G$, $H$: $\mathbb{R} \to S$ satisfy~\eqref{le} and~\eqref{tronca}. Then for $t\in [T/4,T]$, we can write
$$
 \mathcal{N}_{nr}^t[F(t),G(t),H(t)]=\widetilde{\mathcal{E}}_1^t+\mathcal{E}_2^t,
$$
where it holds that, uniformly in $T\geq 1$,
\begin{equation*}
T^{1+2\delta}\sup_{T/4\leq t\leq T}\Vert \widetilde{\mathcal{E}_1}(t)\Vert_S\lesssim 1,\quad
T^{1/10}\sup_{T/4\leq t\leq T}\Vert \mathcal{E}_3(t)\Vert_S\lesssim 1,
\end{equation*}
where $\mathcal{E}_2(t)=\partial_t\mathcal{E}_3(t)$.
Assuming in addition that~\eqref{BA+} holds we have 
\begin{equation*}
T^{1+2\delta}\sup_{T/4\leq t\leq T}\Vert \widetilde{\mathcal{E}_1}(t)\Vert_{S^+}\lesssim 1,
\qquad T^{1/10}\sup_{T/4\leq t\leq T}\Vert \mathcal{E}_3(t)\Vert_{S^+}\lesssim 1.
\end{equation*}
\end{lemma}

\begin{proof} To begin with, let us recall the following estimate
\begin{equation}\label{sp}
\Big\|
\sum_{(q,r,s)\,:\,(p,q,r,s)\in{\mathcal M}}
c^1_qc^2_rc^3_s\Big\|_{\ell^2_{p}}
\lesssim
\min_{\sigma\in\mathfrak{S}_3}\|c^{\sigma(1)}\|_{\ell^2_p}\|c^{\sigma(2)}\|_{\ell^1_p}\|c^{\sigma(3)}\|_{\ell^1_p}\,,
\end{equation}
which a direct consequence of Young's convolution estimates.\medskip

To prove the lemma, we start by decomposing $\mathcal{N}_{nr}^t$ along the non-resonant level sets as follows: let $\varphi \in \mathcal{C}^{\infty}_{0}(\R)$ be such that $\varphi\equiv 1$ near 0, then define
\begin{align}
\mathcal{O}^t_1[f,g,h](\xi)&:=\int_{\mathbb{R}^2}e^{2it\eta\kappa}(1-\varphi(t^{\frac{1}{4}}\eta\kappa))\widehat{f}(\xi-\eta)\overline{\widehat{g}}(\xi-\eta-\kappa)\widehat{h}(\xi-\kappa)d\eta d\kappa,
\nonumber\\
\mathcal{O}^t_2[f,g,h](\xi)&:=\int_{\mathbb{R}^2}e^{2it\eta\kappa}\varphi(t^{\frac{1}{4}}\eta\kappa)\widehat{f}(\xi-\eta)\overline{\widehat{g}}(\xi-\eta-\kappa)\widehat{h}(\xi-\kappa)d\eta d\kappa.\nonumber
\end{align}
 \begin{equation}\label{decomp}
 \mathcal{N}_{nr}^t[F,G,H] = \mathcal{N}_{nr,1}^t[F,G,H]+ \mathcal{N}_{nr,2}^t[F,G,H]
 \end{equation}
\begin{equation*}
\mathcal{F} \mathcal{N}_{nr,j}^t[F,G,H](\xi,p)=
 \sum_{(p,q,r,s)\in\Gamma_{nr}}e^{it\omega} \mathcal{O}^t_j[F_q,G_r,H_s](\xi) , 
\end{equation*} 
where we used the notation $\Gamma_{nr}=\mathcal{M}\backslash \Gamma_0$ for the non resonant terms in $\mathcal{M}$.\medskip

$\bullet$ The first term $\mathcal{N}_{nr,1}^t[F,G,H]$      in~\eqref{decomp} can be controlled exactly as~\cite{HPTV}, this is the content of~\cite[Lemma 3.6]{HPTV}. \medskip    

$\bullet$  We now control the term $\mathcal{N}_{nr,2}^t[F,G,H]$. In the sequel we write $\big\{F,G,H\big\}=\big\{F^{a},F^{b},F^{c}\big\}$ and we assume that~\eqref{tronca} holds true.  Then we  define
\begin{equation*}
\Vert f\Vert_Y:=\Vert\langle x\rangle^{\frac{9}{10}}f\Vert_{L^2_x}+\Vert f\Vert_{H^{\frac{3N}{4}}_x},
\end{equation*}
and we note that  if $N>N_{0}(d,\gamma)$ large enough  we have
\begin{equation}\label{sy}
\sum_{p\in\mathbb{Z}^d}\Vert |p|^{\gamma} F_p\Vert_Y\lesssim \Vert F\Vert_S.
\end{equation}
 Let $t\geq T/4$. By~\cite[Estimate (3.20)]{HPTV} we have the bound
 \begin{equation}\label{CEC}
 \Vert \mathcal{O}^t_2 [F^a,F^b,F^c]\Vert_{L^2_{\xi}}  \lesssim (1+\vert t\vert)^{-1+\delta}\min_{\sigma\in\mathfrak{S}_3}\Vert F^{\sigma(a)}\Vert_{L^2_x}\Vert F^{\sigma(b)}\Vert_{Y}\Vert F^{\sigma(c)}\Vert_{Y}.
\end{equation}

We write
\begin{multline}\label{TIPP}
 e^{it\omega}\mathcal{O}^t_2[f,g,h]=\partial_t\Big( \frac{e^{it\omega}}{i\omega}\mathcal{O}^t_2[f,g,h]\Big)
-\frac{e^{it\omega}}{\omega}\left(\partial_t\mathcal{O}^t_2\right)[f,g,h]
\\
-\frac{e^{it\omega}}{i\omega}\mathcal{O}^t_2[\partial_tf,g,h]
-\frac{e^{it\omega}}{i\omega}\mathcal{O}^t_2[f,\partial_tg,h]-\frac{e^{it\omega}}{i\omega}\mathcal{O}^t_2[f,g,\partial_th],
\end{multline}
where
$$
\left(\partial_t\mathcal{O}^t_2\right)[f,g,h](\xi):=\int_{\mathbb{R}}\partial_t\Big(e^{2it\eta\kappa}\varphi(t^{\frac{1}{4}}\eta\kappa)\Big)
\widehat{f}(\xi-\eta)\overline{\widehat{g}}(\xi-\eta-\kappa)\widehat{h}(\xi-\kappa)d\eta d\kappa.
$$
We define $\mathcal{E}_3$ by
\begin{equation*}
 \mathcal{F}\mathcal{E}_3(\xi,p):= \sum_{(p,q,r,s)\in\Gamma_{nr}}\frac{e^{it\omega}}{i\omega}\mathcal{O}^t_2[F_q,G_r,H_s](\xi).
\end{equation*}
We now estimate the contribution of each term in~\eqref{TIPP}. Here we face an additional difficulty, since $|\omega|$ is not bounded from below by a constant as in~\cite{HPTV}. Therefore we will need Assumption~\ref{As1}.\medskip

$\star$ We first consider the term $ \mathcal{F}_{x}\mathcal{E}_3$.   By~\cite[Lemma~7.4]{HPTV} it is enough to prove that 
\begin{equation*} 
\Vert  \mathcal{E}_{3}\Vert_{L^2_{x,y}}\lesssim (1+\vert t\vert)^{-1+\delta} \min_{\sigma\in\mathfrak{S}_3}\Vert F^{\sigma(a)}\Vert_{L^2_{x,y}}\Vert F^{\sigma(b)}\Vert_S\Vert F^{\sigma(c)}\Vert_S.
\end{equation*}
By symmetry, the previous inequality will be implied by  
\begin{equation}\label{bound1} 
\Vert  \mathcal{E}_{3}\Vert_{L^2_{x,y}}\lesssim (1+\vert t\vert)^{-1+\delta} \Vert F^{a}\Vert_{L^2_{x,y}}\Vert F^{b}\Vert_S\Vert F^{c}\Vert_S,
\end{equation}
and we now prove this estimate. Let  $K\in L^2_{\xi,p}(\mathbb{R}\times \mathbb{Z}^d)$, then 
\begin{eqnarray*} 
\langle K,\mathcal{F}\mathcal{E}_{3} \rangle_{L^2_{\xi,p}\times L^2_{\xi,p}}&\le&\sum_{(p,q,r,s)\in\Gamma_{nr}}
  \big\vert\langle K_p,\frac{1}{\omega}\mathcal{O}^t_2[F^{a}_q,F^{b}_r,F^{c}_s]\rangle_{L^2_\xi\times L^2_\xi}\big\vert\\
  &\le&\sum_{(p,q,r,s)\in\Gamma_{nr}}
   \Vert K_p\Vert_{L^2_\xi}      \big\Vert \frac{1}{\omega}\mathcal{O}^t_2[F^{a}_q,F^{b}_r,F^{c}_s]   \big \Vert_{L^2_\xi}    
  \end{eqnarray*}
Then, using~\eqref{res1}, we can assume that on $\Gamma_{nr}$, $|\omega|\geq c|r|^{-\gamma}$ or  $|\omega|\geq c|s|^{-\gamma}$, therefore by~\eqref{CEC}
\begin{eqnarray*} 
\langle K,\mathcal{F}\mathcal{E}_{3} \rangle_{L^2_{\xi,p}\times L^2_{\xi,p}} 
&   \lesssim&\sum_{(p,q,r,s)\in\Gamma_{nr}}\Vert K_p\Vert_{L^2_\xi}   \big( \| \mathcal{O}^t_2[F^{a}_q,|r|^{\gamma}F^{b}_r,F^{c}_s]\|_{L^{2}_{\xi}}+ \| \mathcal{O}^t_2[F^{a}_q,F^{b}_r,|s|^{\gamma}F^{c}_s]\|_{L^{2}_{\xi}}\big)\\
    &\lesssim& (1+\vert t\vert)^{-1+\delta}\sum_{(p,q,r,s)\in\Gamma_{nr}}\Vert K_p\Vert_{L^2_\xi}     \Vert F^{a}_q\Vert_{L^2_x}\big\Vert |r|^{\gamma}F^{b}_r\big\Vert_{Y}\Vert |s|^{\gamma}F^{c}_s\Vert_{Y} .
\end{eqnarray*}
 The estimate~\eqref{bound1} then follows from an application of~\eqref{sp} and~\eqref{sy}.\medskip

$\star$ Since $(1+\vert t\vert)^{1/4}(\partial_t\mathcal{O}_{2}^{t})$ satisfies similar estimates as $\mathcal{O}_{2}^{t}$, the second term in the right hand-side of~\eqref{TIPP} can be estimated as $\mathcal{E}_{3}$ and is therefore  acceptable.\medskip

$\star$ The contribution of the terms in the second line of~\eqref{TIPP} is estimated as $\mathcal{E}_{3}$ and by using the definition of the $X_{T^\ast}$ norm.\medskip

This ends the estimation of $\mathcal{O}^{t}_{2}$ and the proof of Lemma~\ref{FL}.
\end{proof}

\section{The resonant system} 
In this section, we study the resonant system which dictates the dynamics of~\eqref{NLS}. For $p\in \Z$, we consider the system
\begin{equation}\label{RS}
i\partial_ta_p(t)=\sum_{(p,q,r,s)\in\Gamma_0}a_{q}(t)\overline{a_{r}(t)}a_{s}(t)=:R[a(t),a(t),a(t)]_p.
\end{equation}
This is a Hamiltonian system for the symplectic form
\begin{equation*}
\Omega(\{a_p\},\{b_q\})={\rm Im}\Big[\sum_{p\in\mathbb{Z}^d}\overline{a_p}b_p\Big]={\rm Re}\langle -i\{a_p\},\{b_p\}\rangle_{\ell^2_p\times \ell^2_p}
\end{equation*}
and Hamiltonian
\begin{equation*}
\langle R(a,a,a),a\rangle_{\ell^2_p\times \ell^2_p}=\sum_{(p,q,r,s)\in\Gamma_0}a_p\overline{a_q}a_r\overline{a_s}.
\end{equation*}

The next result gives a geometrical  description of the resonant set, but will not be used in the sequel. 

\begin{lemma} 
The points $(p,q,r,s)\in\Gamma_0$ if and only if $\{p,q,r,s\}$ are the 
successive edges of a rectangle such that the origin belongs to the 
perpendicular bisector hyperplane of one of its vertices.
\end{lemma}
\begin{proof}
The claim follows from elementary geometry: the condition $p-q+r-s = 0$ 
imposes that the four points form a parallelogram. The resonance 
condition implies that the origin belongs to the perpendicular bisectors 
of two parallel vertices. Since these hyperplanes are parallel and 
intersect at the origin, they must coincide, and this implies the 
orthogonality of the vertices.
\end{proof}

\subsection{First integrals and wellposedness}
\begin{lemma}\label{LemDST}
Let $R$ be defined as in~\eqref{RS}. For every sequences $(a^1)_p$, $(a^2)_p$, $(a^3)_p$ indexed by $\mathbb{Z}^d$ with $1\leq d\leq 4$ 
\begin{equation}\label{StriEst2}
\Vert R[a^1,a^2,a^3]\Vert_{\ell^{2}_p}\leq C_d\,
\min_{\sigma \in \mathfrak{S}_{3}} \Vert a^{\sigma(1)}\Vert_{\ell^{2}_p}\Vert a^{\sigma(2)}\Vert_{h^{1}_p}\Vert a^{\sigma(3)}\Vert_{h^{1}_p}.
\end{equation}
and consequently, for any $s \geq 1$,
\begin{equation}\label{StriEst}
\Vert R[a^1,a^2,a^3]\Vert_{h^s_p}\leq C_{\sigma,d}\,
\sum_{\sigma\in\mathfrak{S}_3}\Vert a^{\sigma(1)}\Vert_{h^s_p}\Vert a^{\sigma(2)}\Vert_{h^{1}_p}\Vert a^{\sigma(3)}\Vert_{h^{1}_p}.
\end{equation}
\end{lemma}

\begin{proof} 
Let $a^0\in \ell^{2}_p$ and compute 
\begin{equation*}
\vert \<a^0, R[a^1,a^2,a^3]\>_{\ell^{2}_{p}}\vert \leq \sum_{(p,q,r,s)\in \Gamma_0}\vert a^0_p\vert  \vert a^1_q\vert \vert a^2_r\vert \vert a^3_s\vert .
\end{equation*}
Then~\eqref{StriEst2} follows from~\cite[Lemma 7.1]{HPTV}, since 
\begin{equation*}
\sum_{(p,q,r,s) \in \Gamma_{0}} \vert a^{0}_{p} \vert \vert a^{1}_{q} \vert \vert a^{2}_{r} \vert \vert a^{3}_{s} \vert \leq C \sum_{\substack{p+r = q+s\\ \vert p\vert^{2}+\vert r \vert^{2} = \vert q \vert^{2} +\vert s \vert^{2}}} \vert a^{0}_{p} \vert \vert a^{1}_{q} \vert \vert a^{2}_{r} \vert \vert a^{3}_{s} \vert\,.
\end{equation*}
Estimate~\eqref{StriEst} comes from the Leibniz rule proved in~\cite[Lemma 7.4]{HPTV}.
\end{proof}
\medskip

The resonant system is well defined for initial data in $h^{1}_p$:
\begin{lemma} 
Let $1\leq d\leq 4$. For any $ a(0)\in h^{1}_p$, there exists a unique global solution ${u\in \mathcal{C}^1(\mathbb{R}; h^{1}_p)}$ of~\eqref{RS}.
In addition, if $a(0)\in h^s_p$ for $s \geq1$, then the solution belongs to $\mathcal{C}^1(\mathbb{R};h^s_p)$, and in this case the $h^s_p$ norms are constants of motion. 
\end{lemma}

It would be interesting to know whether the equation~\eqref{RS} is well posed in $H^{1}$  for any $d\geq 1$. One can also ask if the equation~\eqref{RS} is wellposed in $L^{2}$. Using the conservation laws, it is quite easy to construct solutions in $L^{2}$ by the means of compactness arguments, but the uniqueness is unclear.

\begin{proof}
 From Lemma \ref{LemDST} we observe that the mapping $ a \mapsto R\left[ a,a,a\right]$ is locally Lipschitz in $h^{1}_{p}$ uniformly on bounded sets. A contraction mapping argument gives local well-posedness in $h^{s}_{p}$ for any $s \geq1$, which is extended to a global statement in $h^{1}_{p}$ by the conservation of mass
 \begin{equation}\label{conservmass}
 \mathrm{mass}(a) = \sum_{p \in \Z^{d}} \vert a_{p} \vert^{2}
 \end{equation}
 and the conservation of energy 
  \begin{equation}\label{conservener}
 \mathrm{energy}(a) = \sum_{p \in \Z^{d}} \vert p \vert^{2} \vert a_{p} \vert^{2}.
 \end{equation}
 Finally observe that the quantities 
\begin{equation*} 
 A_{N}(a) =  \Big( \sum_{\substack{p \in \Z^{d}\\ \vert p\vert^{2}=N}} \vert a_{p}\vert^{2} \Big)^{{\frac 12}}
 \end{equation*}
 are also quantities conserved by the dynamics of~\eqref{RS}: indeed, we have
  \begin{eqnarray*}
 \frac{d}{dt} \big( A^{2}_{N}(a) \big) &=& \sum_{\substack{p \in \Z^{d}\\ \vert p \vert^{2} = N}} \big(\partial_{t} a_{p} \overline{a_{p}} + a_{p} \partial_{t}\overline{a_{p}}\big)\\
 &  = & \sum_{\substack{p \in \Z^{d}\\ \vert p \vert^{2} = N}} \Big( -i \sum_{\substack{(q,r,s),\\ (p,q,r,s) \in \Gamma_{0}}} \overline{a_{p}}a_{q} \overline{a_{r}} a_{s} + i \sum_{\substack{(q,r,s),\\ (p,q,r,s) \in \Gamma_{0}}}  a_{p}\overline{a_{q}}a_{r} \overline{a_{s}} \Big)\\
 &=& - 2 \Im  \sum_{\substack{p \in \Z^{d}\\ \vert p \vert^{2} = N}} \sum_{\substack{(q,r,s),\\ (p,q,r,s) \in \Gamma_{0}}} \overline{a_{p}}a_{q} \overline{a_{r}} a_{s},
 \end{eqnarray*}
 and this last sum is real: for $p$ such that $\vert p \vert^{2}=N$, if $(q,r,s)$ are such that $(p,q,r,s) \in \Gamma_{0}$, then $\vert q \vert^{2}=N$ for example (the case $\vert s \vert^{2}=N$ is similar), and the sum also includes the term $a_{p}\overline{a_{q}}a_{r} \overline{a_{s}}$ corresponding to the rectangle $(q,p,s,r) \in \Gamma_{0}$, which gives the preservation of $A_{N}(a)$. This implies in particular that all  $h^{s}_{p}$ norms are preserved for $s\geq0$:

 \begin{equation*}
\Vert a \Vert_{h^s_p}^2:=\sum_{p\in\mathbb{Z}^d}\left[1+\vert p\vert^2\right]^s\vert a_p\vert^2 \sum_{N \geq0} \left[1+ N^2\right]^s \sum_{\substack{p \in \Z^{d}\\ \vert p\vert^{2}=N}}\vert a_p\vert^2 = \sum_{N \geq0} \left[1+ N^2\right]^s A_{N}^{2}(a),
\end{equation*}  
which completes the proof.
\end{proof}

\subsection{Estimation of solutions to the resonant system}

\begin{lemma}\label{gAdmi}
 The equation~\eqref{RSS} is wellposed for initial conditions in $Z$ or in $S$. Moreover,  for all   $N \geq 0$ and all $t\geq 0$
\begin{eqnarray} \label{Sol1}
\Vert G( t)\Vert_Z&=&\Vert G_0\Vert_Z,\nonumber\\
\Vert G( t)\Vert_{H^{N}_{x,y}} &=& \Vert G_{0}\Vert_{H^{N}_{x,y}}.
\end{eqnarray}
\end{lemma}

\begin{proof} 
The existence proof follows the lines of~\cite[Lemma 4.3]{HPTV}. The two  equalities in~\eqref{Sol1} directly follow from the first integrals~\eqref{conservmass} and~\eqref{conservener}, the dependance in $\xi$ being parametric in the system~\eqref{RSS}.
\end{proof}

The  result of Corollary\;\ref{coro14} is a direct consequence of \eqref{Sol1}.
 

\end{document}